\documentclass{article}

\usepackage{amssymb,amsmath,eufrak,amsthm,amscd}

\newcommand{\Spec}{\operatorname{Spec}}

\newcommand{\diffspec}{\operatorname{Spec^\Delta}}
\newcommand{\diffmax}{\operatorname{Max^\Delta}}
\newcommand{\diffsmax}{\operatorname{SMax^\Delta}}

\newtheorem{theorem}{Theorem}
\newtheorem{lemma}[theorem]{Lemma}

\newtheorem{corollary}[theorem]{Corollary}
\newtheorem{statement}[theorem]{Statement}
\newtheorem*{statement*}{Statement}
\newtheorem*{theorem*}{Theorem}
\newtheorem*{lemma*}{Lemma}
\newtheorem*{fact*}{Fact}

\theoremstyle{definition}

\newtheorem*{definition*}{Definition}

\newtheorem*{example*}{Example}
\newtheorem{remark}[theorem]{Remark}
\newtheorem*{remark*}{Remark}

\theoremstyle{remark}

\title{Inheritance of Properties of Spectra}
\author{D.\,V.~Trushin}
\date{}

\begin{document}

\maketitle

\begin{abstract}
A mechanism for the inheritance of properties of spectra by
differential spectra is developed and applied to prove geometric
properties of morphisms of differential algebraic varieties.
\end{abstract}

\section{Introduction}

This paper presents a fundamentally new differential algebraic
technique, called inheritance. We develop an apparatus which makes
it possible to automatically derive certain facts of differential
algebra from their commutative counterparts. Roughly speaking, the
problem is divided into two parts, determining automatic dependences
and proving theorems for the nondifferential case.

The first step is completely described by
Statement~\ref{prop_lemma}. The reduction of a differential problem
to a nondifferential one follows from Statement~\ref{lemmaStuct}.
The main result obtained by our method is Theorem~\ref{GD_theorem}.
Being combined with geometric facts, such as
Statements~\ref{theoremUp} and~\ref{theoremDown},
Theorem~\ref{GD_theorem} makes it possible to prove new interesting
geometric properties of differential algebraic varieties, in
particular, Statement~\ref{open_var} and Theorem~\ref{open_group}.
The latter implies that all dominant morphisms of differential
algebraic groups are open surjective maps.

It is worth mentioning that other authors working in this direction
used mainly the technique of characteristic sets (see, e.g.,
Cassidy's paper~\cite{Cs} or the classical books~\cite{Kl}
and~\cite{Kl2} by Kolchin). We supplement the well-known technique
by ideas of the inheritance method, which yields deeper geometric
results. Some results, such as the local openness of morphisms for
differential algebraic varieties (Statement~\ref{open_var}) are
interesting in themselves, because they demonstrate new geometric
effects; it is these effects that are the object of study in
differential algebraic geometry.

The most nontrivial part of proofs is verifying the relationship
between the flatness of a commutative ring and the geometry of its
differential spectrum described in Theorem~\ref{GD_theorem}. The
reduction of the initial problem to a nondifferential problem by
methods described in Sec.~4 makes it possible to establish more
subtle relationship between differential spectra. However, it
remains unclear how to manage without flatness, which does not
depend on the differential ring structure.

\section{Organization of the paper by sections}

Section~\ref{sec_3} contains a brief review of definitions. In
Sec.~\ref{sec_4}, the main method is explained and basic inherited
properties are introduced. Section~\ref{sec_5} contains direct
corollaries, such as the going-down and going-up theorems, of the
statements proved in Sec.~\ref{sec_4}. To deal with the case of main
interest for us, that is, the case of differentially finitely
generated algebras over a field of characteristic zero, we need a
technical result (Statement~\ref{lemmaStuct}), which we prove in
Sec.~\ref{sec_6}. In the next, seventh, section, its application is
exemplified by proving known facts. In Sec.~\ref{sec_8}, results on
rings are translated into the geometric language of spectra, after
which, in Sec.~\ref{sec_9}, the main theorem is proved in
ring-theoretic and topological terms. Section~\ref{sec_10} is
devoted to locally closed points of differential spectra, which play
a role similar to that played by maximal spectra in algebraic
geometry. In Sec.~\ref{sec_11}, after some necessary preparations,
geometric corollaries for differential algebraic varieties and
groups are stated.

\section{Definitions and explanations}\label{sec_3}

Throughout the paper, by a ring we mean an associative commutative
ring with identity. All ring homomorphisms are identity-preserving.
We assume the knowledge of the material contained in the
books~\cite{AM} and~\cite{Kl}; any notions and notation used but not
defined in the paper can be found in these books.

We briefly explain some notation below. The set of derivations of a
ring is denoted by $\Delta$, and instead of the term ``differential
ring'' we often use the term ``$\Delta$-ring''; other terms are
treated in a similar manner. If $A$ and $B$ are differential rings,
$f\colon A\to B$ is a differential homomorphism between them, and
$\frak a$ and $\frak b$ are ideals in $A$ and $B$, respectively,
then the ideal $Bf(\frak a)$ generated by the image of $f(\frak a)$
in $B$ is called an {\it extension} of the ideal $\frak a$ and
denoted by $\frak a^e$; the ideal $f^{-1}(\frak b)$ is called a {\it
contraction} of $\frak b$ and denoted by $\frak b^c$. Contractions
and extensions of differential ideals are differential ideals. We
denote the set of prime ideals in a ring $A$ by $\Spec A$ and call
it the {\it spectrum} of $A$. Similarly, the set of prime
differential ideals in a differential ring $A$ is denoted by
$\diffspec A$ and called the {\it differential spectrum} of this
ring. By a {\it simple} differential ring we always understand a
differential ring containing no nontrivial differential ideals.

By $f^*$ we denote the map from the spectrum of $B$ to the spectrum
of $A$ defined by $f^*(\frak q)=f^{-1}(\frak q)$. We write
$f^*_\Delta$ for the restriction of $f^*$ to the differential
spectrum. We assume the spaces $\Spec A$ and $\diffspec A$ to be
endowed with the Zariski topology. A differential ring $A$ is called
a {\it Keigher ring} if the radical of any differential ideal $\frak
a$ of $A$ is a differential ideals as well. A differential ring $A$
is said to be a {\it Ritt algebra} if $A$ is an algebra over
$\mathbb Q$. Note that any Ritt algebra is a Keigher ring~\cite{Kpl}

\section{Inheritance of properties}\label{sec_4}

We understand inheritance of properties of spectra as follows. Let
$f\colon A\to B$ be a differential homomorphism of Keigher rings.
Consider a pair of properties
\begin{itemize}
\item $A1$, which is a property of $f$ treated as a homomorphism,
and
\item $A2$, which is a property of $f$ treated as a differential
homomorphism
\end{itemize}
such that ($A1$) implies ($A2$). The idea is to find an appropriate
pair of such properties and reduce a differential fact to a
nondifferential one. As we shall demonstrate in what follows, this
approach is indeed fruitful. We start with looking for suitable
pairs of properties.

The first observation concerning a relationship between prime
spectra and prime differential spectra is as follows.

\begin{statement}\label{lemmaIdeal}
Let $A\to B$ be a differential homomorphism of Keigher rings, and
let $\frak p$ be a prime differential ideal in $A$. Then the
following conditions are equivalent:
\begin{enumerate}
\item $\frak p^{ec}=\frak p$
\item $(f^*)^{-1}(\frak p)\neq\emptyset$
\item $(f^*_{\Delta})^{-1}(\frak p)\neq\emptyset$
\end{enumerate}
\end{statement}
\begin{proof}
The implications (3)$\Rightarrow$(1)$\Rightarrow$(2) follow in an
obvious way from known facts of commutative algebra~\cite[Chap.~3,
Sec.~2, Proposition~3.16]{AM}. Let us show that (2)$\Rightarrow$(3).
According~\cite[Chap.~3, Sec.~2, Exercise~21 (IV)]{AM}, the fiber
$(f^*)^{-1}(\frak p)$ is naturally homeomorphic to $\Spec(B_{\frak
p}/\frak p B_\frak p)$; moreover, it is easy to see that the
differential fiber $(f^*_{\Delta})^{-1}(\frak p)$ is naturally
homeomorphic to $\diffspec(B_{\frak p}/\frak p B_\frak p)$. Since
the fiber is nonempty, it follows that the ring $B_{\frak p}/\frak p
B_\frak p$ is nonzero, and since this ring is Keigher, it follows
that its differential spectrum is nonempty.
\end{proof}

Now, we need a couple more definitions. Let $f\colon A\to B$ be a
ring homomorphism.

We say that $f$ has the {\it going-up property} if, given a chain
$\mathfrak{p}_1\subseteq\ldots\subseteq\mathfrak{p}_n$ with $\frak
p_i\in \Spec f(A)$, any nonempty chain
$\mathfrak{q}_1\subseteq\ldots\subseteq\mathfrak{q}_m$ $(m<n)$ and
$\frak q_i\in \Spec B$ satisfying the condition $\mathfrak{q}_i\cap
f(A)=\mathfrak{p}_i$ for  $(1\leqslant i\leqslant m)$ can be
extended to a chain
$\mathfrak{q}_1\subseteq\ldots\subseteq\mathfrak{q}_n$ satisfying
the condition $\mathfrak{q}_i\cap f(A)=\mathfrak{p}_i$ for
$1\leqslant i\leqslant n$.

We say that $f$ has the {\it going-down property} if, given any
chain $\mathfrak{p}_1\supseteq\ldots\supseteq\mathfrak{p}_n$ with
$\frak p_i\in\Spec f(A)$, any nonempty chain
$\mathfrak{q}_1\supseteq\ldots\supseteq\mathfrak{q}_m$ with $(m<n)$
and $\frak q_i\in \Spec B$ satisfying the condition
$\mathfrak{q}_i\cap f(A)=\mathfrak{p}_i$ for $(1\leqslant i\leqslant
m)$ can be extended to a chain
$\mathfrak{q}_1\supseteq\ldots\supseteq\mathfrak{q}_n$ satisfying
the condition $\mathfrak{q}_i\cap f(A)=\mathfrak{p}_i$ for
$1\leqslant i\leqslant n$.

Replacing chains of ideals in these definitions by chains of
differential ideals, we obtain definitions of the {\it going-up} and
the {\it going-down property for differential ideals}.

\begin{statement}\label{prop_lemma}
Let $f\colon A\to B$ be a differential homomorphism of Keigher
rings, and let $\frak p$ be a prime differential ideal in $A$. Then
the following assertions hold:
\begin{enumerate}
\item\label{fiber_prop} If $(f^*)^{-1}(\frak p)\neq\varnothing$, then $(f^*_{\Delta})^{-1}(\frak
p)\neq\varnothing$;
\item\label{surj_prop} if $f^*$ is surjective, then so is $f^*_{\Delta}$.
\item\label{GUP_prop} if $f$ has the going-up property, then $f$
has the going-up property for differential ideals;
\item\label{GDP_prop} if $f$ has the going-down property, then $f$
has the going-down property for differential ideals.
\end{enumerate}
\end{statement}
\begin{proof}

Assertion~(\ref{fiber_prop}) is a special case of
Statement~\ref{lemmaIdeal}. The second assertion follows directly
from the first applied to all prime differential ideals of $A$.

To prove~(\ref{GUP_prop}), it suffices to consider two prime ideals
$\frak p\subseteq\frak p'$ in $A$ and a prime ideal $\frak q$ in $B$
satisfying the condition $\frak q^c=\frak p$. Applying
Statement~\ref{lemmaIdeal} to the pair of rings $A/\frak p\subseteq
B/\frak q$ and the ideal $\frak p'$, we obtain the required
implication.

Similarly, it is sufficient to prove~(\ref{GDP_prop}) for two prime
ideals $\frak p'\subseteq\frak p$ in $A$ and a prime ideal $\frak q$
in $B$ satisfying the condition $\frak q^c=\frak p$. Applying
Statement~\ref{lemmaIdeal} to the rings $A_{\frak p}\subseteq
B_{\frak q}$ and the ideal $\frak p'$, we obtain the required
implication.
\end{proof}

\section{First applications}\label{sec_5}

Now, we concentrate on the application of the found pairs of
properties. By way of demonstration, we generalize the well-known
Cohen-Seidenberg going-up and going-down theorems
(see~\cite[Chap.~5]{AM}) to Keigher differential rings.

\begin{theorem}[``going-up'']
Suppose that $f\colon A\to B$ is a differential homomorphism of
Keigher rings and $B$ is integral over $A$. Then $f$ has the
going-up property for differential ideals.
\end{theorem}
\begin{proof}
It follows from the going-up Theorem~\cite[Chap.~5, Sec.~2,
Theorem~5.11]{AM} that $f$ has the going-up property. According to
Statement~\ref{prop_lemma}, the homomorphism $f$ has the going-up
property for differential ideals.
\end{proof}

\begin{theorem}[``going-down'']
Suppose that $f\colon A\to B$ is a differential homomorphism of
Keigher rings, $B$ is an integral domain, $f(A)$ is integrally
closed, and $B$ is integral over $A$. Then $f$ has the going-down
property for differential ideals.
\end{theorem}
\begin{proof}
It follows from the going-down theorem~\cite[Chap.~5, Sec.~3,
Theorem~5.16]{AM} that $f$ has the going-down property. According to
Statement~\ref{prop_lemma}, $f$ has the going-down property for
differential ideals.
\end{proof}

Proofs of these two theorems for the case of Ritt algebras are
contained in~\cite{Tr}.

The going-up and going-down theorems are not the most interesting
applications of the pairs of properties specified above. The most
interesting case is that of differentially finitely generated
algebra over a field of characteristic zero. To deal with them, we
need a special mechanism for reducing differential problems to
nondifferential ones; such a mechanism is presented in the nest
section.

\section{A structure result}\label{sec_6}

The following statement is merely a somewhat different view of
characteristic sets for differential ideals.

\begin{statement}\label{lemmaStuct}

Suppose that $A$ is a Ritt algebra and $B$ is a differentially
finitely generated integral algebra over $A$. Then there exists an
element $s$ in $B$ for which $B_s=C[y_{\alpha}]$, where $C$ is a
finitely generated algebra over $A$ and $y_\alpha$ is an at most
countable set of variables algebraically independent over $C$.
\end{statement}
\begin{proof}

Let us represent the ring $B$ as the quotient of the ring of
differential polynomials over $A$ by an ideals $\frak p$, i.e., as
$$
B=A\{y_1,\ldots,y_n\}/\frak p.
$$
Consider a characteristic set $F=\{\,f_1,\ldots,f_k\,\}$  of $\frak
p$ with respect of some ranking. We denote the product of all
initials and separants of the set $F$ by $s$. Let $W$ be the set of
all variables on which polynomials from $F$ depend, and let $Z$ be
the set of variables that are not proper derivatives of leaders of
elements of $F$ and do not belong to $W$. It follows
from~\cite[Chap.~I, Sec.~9, Proposition~1]{Kl} that $B_s$ can be
presented in the required form; for the algebra $C$ we take the
subalgebra generated over $A$ by the set $W$ and $1/s$ and for the
family $\{y_\alpha\}$, the images of the set $Z$; we have
$$
C=A[W]_s/(f_1,\ldots,f_k).
$$
\end{proof}

\section{Demonstration}\label{sec_7}

In this section, we demonstrate the application of the method by
more instructive examples. The facts given below are well known;
their various forms can be found in many works. We give these fact
only for demonstration. Thus, Lemma~2 in~~\cite[Chap.~III,
Sec.~3]{Kl2} is an analog of the following theorem in the sense that
we state this theorem for prime differential spectra rather than for
specializations. Yet another version of this result can be found
in~\cite[Chap.~1, Sec.~6, Proposition~2]{Cs}.

\begin{theorem}\label{maintheorem}
Suppose that $A\subseteq B$ are differential integral Ritt algebras
and the ring $B$ is differentially finitely generated over $A$. Then
the ring $A$ contains an element $s$ for which the map
$$
\diffspec B_s\rightarrow\diffspec A_s
$$
is surjective.
\end{theorem}
\begin{proof}

It sufficient to prove the surjectivity of this map for the ring
$B_u$, where $u$ is an element of $B$. Choose an element $u$ as in
Statement~\ref{lemmaStuct}. Let as show that the corresponding map
on prime spectra is surjective; this is sufficient by
Statement~2(2). Indeed, in the notation of
Statement~\ref{lemmaStuct}, the spectrum of $B_u$ is surjectively
mapped onto the spectrum of $C$. It remains to apply the well-known
fact from commutative algebra that there exists an element $s$ in
$A$ for which the map from the spectrum of $C_s$ onto the spectrum
of $A_s$ is surjective (see~\cite[Chap.~5, Exercise~20 and
Theorem~5.10]{AM})
\end{proof}

\begin{corollary}
Let $B$ be a simple differentially finitely generated algebra over a
differential field of characteristic zero, and let $A\subseteq B$ be
an any subalgebra over this field. Then there exists an element $s$
in $A$ for which $A_s$ is simple.
\end{corollary}

\begin{remark}
In particular, let $K$ be the mentioned field. Then taking $K\{y\}$,
where $y$ is any element of $B$, for the algebra $A$, we see that
any element of a simple differentially finitely generated algebra
over the field is differentially dependent.
\end{remark}

The following assertion is an version of Proposition~7
in~\cite[Chap.~III, Sec.~10]{Kl}

\begin{corollary}\label{ConstColl}
Suppose that $B$ is a simple differentially finitely generated
algebra over a differential field $K$ of characteristic zero, $F$ is
its fraction field, and $C_K$, $C_B$, and $C_F$ are the
corresponding subrings of constants. Then the subring $C_B$ and
$C_F$ coincide and, therefore, are fields; moreover, the extension
of fields  $C_B/C_K$ is algebraic.
\end{corollary}
\begin{proof}
Let us prove that $C_B=C_F$. Take any element $a\in C_F$. The
fraction ideal
$$
(B:a)=\{\,r\in B\mid ar\in B \,\}
$$
is a nontrivial differential ideal; therefore, by virtue of
simplicity, $B$ contains the identity element, i.e., $a$ belongs to
$B$.

Let us show that any element $y$ of $C_F$ is algebraic over $C_K$.
Theorem~\ref{maintheorem} applied to the pair of rings $K\{y\}=K[y]$
and $B$ implies the existence of an element $s\in K[y]$ for which
$K[y]_s$ is a simple differential ring. Therefore, the element $y$
cannot be algebraically independent.
\end{proof}

\begin{remark}
In particular, if the subfield of constants in the field $K$ is
algebraically closed, then no new constants arise in any simple
differentially finitely generated algebra over the field $K$.
\end{remark}

Thus, the technique presented above makes it possible to obtain new
proofs of known facts (both corollaries can be found
in~\cite[Chap.~3, Sec.~10, Proposition~7]{Kl}). What is more
important, this technique also makes possible to strongly advance
the understanding of the geometry of differential spectra and obtain
quite new results. The rest of the paper is devoted to accomplishing
this task.

\section{Geometric analogs}\label{sec_8}

First, we recall the necessary definitions.

A subset $E$ of a topological space $X$ is said to be {\it
constructible} if it can be presented as $E=U_1\cap
V_1\cup\ldots\cup U_n\cap V_n$, where the sets $U_i$ are open and
the sets $V_i$ are closed.

If a map $f\colon Y\to X$ of topological spaces takes any closed set
$E\subseteq Y$ to a closed set $f(E)$, then the map $f$ is said to
be closed.

If a map $f\colon Y\to X$ of topological spaces takes any open set
$E\subseteq Y$  to an open set $f(E)$, then the map $f$ is said to
be open.

\begin{statement}\label{constructtheorem}
Suppose that $A$ is a Ritt algebra, $B$ is differentially finitely
generated over $A$, and $\diffspec A$ is a Noetherian topological
space. Let $f$ denote the natural embedding of the ring $A$ into
$B$, and let $X$ and $Y$ be topological spaces defined by
$X=\diffspec A$ and $Y=\diffspec B$. Then the set $f^*_\Delta(E)$ is
constructible for any constructible set $E\subseteq Y$.
\end{statement}
\begin{proof}

We use the following criterion: A set $E$ of a Noetherian
topological space is constructible if and only if, for any
irreducible closed set $X_0$, either $\overline{E\cap X_0}\neq X_0$
or else $E\cap X_0$ contains a nonempty open subset of $X_0$
(see~\cite[Chap.~7, Exercise~21]{AM}). The required assertion
follows from this criterion and Theorem~\ref{maintheorem}. Indeed,
it sufficient to consider the case $E=U\cap C$, where $U$ is open
and $C$ is closed in $Y$. Replacing the ring $A$ by its homomorphic
image, we can assume that $E$ is open in $Y$. Since $Y$ is
Noetherian, it follows that $E$ is quasi-compact and, therefore, can
be covered by a finitely many sets of the from $\diffspec B_g$. We
have reduced the proof to the case $E=Y$. The constructibility of
$f^*_{\Delta}(Y)$ is proved by the criterion mentioned above. Let
$X_0$ be an irreducible closed subset in $X$ for which the
intersection $f^*_{\Delta}(Y)\cap X_0$ is dense in $X_0$. Taking the
restriction of $f^*_\Delta$ to $(f^*_{\Delta})^{-1}(X_0)$, we can
assume that $X_0=X$. This reduces the proof to the case where $A$ is
an integral domain and the map $f$ is injective. Let
$Y_1,\ldots,Y_n$ be the irreducible components of $Y$. It suffices
to show that the image $f^*_\Delta(Y_i)$ contains a nonempty open
subset of $X$ for some $i$. It remains to apply
Theorem~\ref{maintheorem} to complete the proof.
\end{proof}

\begin{lemma}\label{min_prime_lemma}
Suppose that $f\colon A\to B$ is a ring homomorphism and $\frak p$
is a minimal prime ideal containing $\ker f$. Then the fibre of
$f^*$ over $\frak p$ is nonempty.
\end{lemma}
\begin{proof}
Passing from $A$ to $A/\ker f$, we can assume that $f$ is injective.
Let $S=A\setminus\frak p$. Consider the exact sequence of rings
$$
0\to A\to B
$$
Localizing by $S$, we obtain the exact sequence
$$
0\to S^{-1}A\to S^{-1}B.
$$
It follows that the ring $S^{-1}B$ is nonzero. By virtue of the
minimality of $\frak p$, the spectrum of this ring is naturally
homeomorphic to the fiver over $\frak p$.
\end{proof}

\begin{statement}\label{theoremUp}
Let $f\colon A\to B$ be a homomorphism of Keigher differential
rings. Then the following conditions are equivalent:
\begin{enumerate}
\item\label{th101} $f$ has the going-up property for differential
ideals;
\item\label{th102} $f^*_{\Delta}$ is closed;
\item\label{th103} for any $\frak q\in\diffspec B$, the map $f^*_{\Delta}\colon\diffspec(B/\frak
q)\to\diffspec(A/\frak q^c)$ is surjective.
\end{enumerate}
\end{statement}
\begin{proof}
The equivalence of~(\ref{th101}) and~(\ref{th103}) follows by
definition. Let us prove that~(\ref{th102}) implies~(\ref{th103}).
It suffices to observe that the closure of the set
$f^*_{\Delta}(\diffspec(B/\frak q))$ is the set $\diffspec(A/\frak
p)$, where $\frak p=\frak q^c$; this, together with the closedness
of the map $f^*_{\Delta}$, implies the surjectivity of this map.

To prove that~(\ref{th101}) implies~(\ref{th102}), we apply
Lemma~\ref{min_prime_lemma}. Suppose that $\frak b$ is a
differential ideal in $B$ and $\frak a=\frak b^c$ is its contraction
to $A$. Let us show that, for any prime differential ideal $\frak p$
containing $\frak a$, there exists a prime ideal $\frak q$ (and
hence, there exists a prime differential ideal, see
Statement~\ref{prop_lemma}) containing $\frak b$ and contracted to
$\frak p$. If $\frak p$ is a minimal prime differential ideals
containing $\frak a$, then the required assertion follows from
Lemma~\ref{min_prime_lemma} applied to the pair of rings $A/\frak
a\to B/\frak b$. If $\frak p$ is an arbitrary prime differential
ideal containing $\frak a$, then $\frak p$ contains a minimal such
ideal $\frak p'$. By the above consideration, for this ideal $\frak
p'$, there exists an ideal $\frak q'$ such that $\frak q'^c=\frak
p'$. It remains to apply the going-up property.
\end{proof}

\begin{statement}\label{theoremDown}
Suppose that $A$ is a Ritt algebra, $B$ is differentially finitely
generated over $A$, and $\diffspec A$ is a Noetherian topological
space. Then the following conditions are equivalent:
\begin{enumerate}
\item\label{th201} $f$ has the going-down property for differential
ideals;
\item\label{th202} $f^*_{\Delta}$ is an open map;
\item\label{th203} for any $\frak q\in\diffspec B$, the map $f^*_{\Delta}\colon\diffspec(B_{\frak
q})\to\diffspec(A_{\frak q^c})$ is surjective.
\end{enumerate}
\end{statement}
\begin{proof}
The equivalence of~(\ref{th201}) and~(\ref{th203}) follows by
definition. Let us prove that~(\ref{th202}) implies~(\ref{th203}).
First, we denote the ideal $\frak q^c$ by $\frak p$. Now, note that
$B_\frak q$ is the direct limit of the rings $B_t$ over $t\in
B\setminus\frak q$. Clearly, for Keigher rings, we have
$$
f^*(\diffspec(B_\frak q))=\bigcap\limits_t f^*_{\Delta}(\diffspec
(B_t));
$$
to show this, it suffices to apply Statement~\ref{prop_lemma} to
Exercise~26 in~\cite[Chap.~3]{AM}. Since $\diffspec (B_t)$ is an
open neighborhood of $\frak q$ and the map $f^*_{\Delta}$ is open,
it follows that the set $f^*(\diffspec (B_t))$ is an open
neighborhood of $\frak p$ and, therefore, contains $\diffspec
(A_\frak p)$.

To prove that~(\ref{th201}) implies~(\ref{th202}), we use the
following criterion: A subset $E$ of a Noetherian topological space
is open if and only if, for any irreducible closed set $X_0$ in $X$,
either $E\cap X_0=\emptyset$ or $E\cap X_0$ contains a nonempty open
subset of $X_0$ (see~\cite[Chap.~7, Exercise~22]{AM}). As in
Statement~\ref{constructtheorem}, this reduces the proof to showing
that $E=f^*_{\Delta}(Y)$ is open in $X$. The going-down property
means that if $\frak p\in E$ and $\frak p'\subseteq\frak p$, then
$\frak p'\in E$; in other words, if $X_0$ is an irreducible closed
subset of $X$ intersecting $E$, then $E\cap X_0$ is dense in $X_0$.
Obviously, any constructible set contains a set open in its closure;
therefore, the set $E\cap X_0$ contains a set open in $X_0$. It
follows from the criterion mentioned above that $E$ is open in $X$.
\end{proof}

\section{Inheritance of geometric properties}\label{sec_9}

Using Statements~\ref{prop_lemma} and~\ref{theoremDown}, we can
significantly strengthen Theorem~\ref{maintheorem}.

\begin{theorem}\label{GD_theorem}
Suppose that $A\subseteq B$ are differential integral Ritt algebras
and the ring $B$ is differentially finitely generated over $A$. Then
$B$ contains an element $s$ for which the embedding $A\subseteq B_s$
has the going-down property for differential ideals.
\end{theorem}
\begin{proof}
Choose an element $u$ as in Statement~\ref{lemmaStuct}. The algebra
$B_u$ is a free $C$-module, therefore, a flat $C$-module. Since the
algebra $C$ is finitely generated over $A$, it follows from
Theorem~52 in~\cite[Chap.~8, Sec.~22]{Mu} that there exists an
element $t$ in $A$ for which the algebra $C_t$ is a free
$A_t$-module, i.e., $B_{tu}$ is a flat $A_t$-algebra. Moreover,
$A_t$ is flat $A$-algebra. Now, using Exercise~11
in~\cite[Chap.~5]{AM}, we see that the map of ordinary spectra has
the going-down property, which implies the going-down property for
differential ideals by Statement~\ref{prop_lemma}.
\end{proof}

\begin{corollary}\label{opencorollary}
Suppose that $A\subseteq B$ are differential integral Ritt algebras
and the ring $B$ is differentially finitely generated over $A$.
Suppose also that the differential spectrum of $A$ is Noetherian.
Then $B$ has an element $s$ for which the map $\diffspec
B_s\to\diffspec A$ is open.
\end{corollary}
\begin{proof}
The required assertion follows readily from the preceding theorem
and Statement~\ref{theoremDown}.
\end{proof}

\section{Locally closed points of differential spectra}\label{sec_10}

We denote the set of maximal differential ideals of a ring $A$ by
$\diffmax A$. As in the case of nondifferential rings, this set is
not necessarily dense, all the less very dense, in $\diffspec A$
(see~\cite[Chap.~5, Exercise~26]{AM} for the definitions).Consider
the subset in the differential spectrum defined by
$$
\diffsmax A=\{\frak p\in\diffspec A \mid \exists s\in A:\, (A/\frak
p)_s \mbox{ is simple}\}
$$
This is the set of all locally closed points of the prime
differential spectrum, which is a topological subspace of the
differential spectrum. Our main objective is to show that such sets
are a suitable substitute for differential spectra in the case of
differentially finitely generated algebras over a field of
characteristic zero, and they play a role similar to that played by
maximal spectra in finitely generated algebras.

\begin{statement}\label{dencetheorem}
If $A$ is a differential Keigher ring, then the set $\diffsmax A$ is
a very dense subset in $\diffspec A$.
\end{statement}
\begin{proof}
It is sufficient to show that any prime differential ideal can be
represented as an intersection of ideals from $\diffsmax A$. Indeed,
for an ideal $\frak p$ belonging to $\diffsmax A$, there is nothing
to prove; if an ideal $\frak p$ does not belong bo this set, then,
for any element $t\in A\setminus\frak p$, there exists a maximal
prime differential ideal $\frak q_t$ such that $\frak p\subset \frak
q_t$ and $t\notin\frak q_t$. By construction, we have $\frak q_t\in
\diffsmax A$ and $\frak p=\bigcap\limits_{t\in A\setminus\frak
p}\frak q_t$.
\end{proof}

From now on, we fix a differential field $K$ of characteristic zero.

\begin{statement}\label{smaxcorrect}
If $A$ and $B$ are differentially finitely generated algebras over
the field $K$ and $f\colon A\to B$ is a differential homomorphism,
then
$$
f^*_{\Delta}(\diffsmax B)\subseteq \diffsmax A.
$$
\end{statement}
\begin{proof}
This follows readily from Theorem~\ref{maintheorem}.
\end{proof}

Thus, in the case of differentially finitely generated algebras over
a field, the restriction of the ideal contraction map to the set of
locally closed points of the differential spectrum is well defined.
Moreover, this set has the following universal property.

Let $A$ be a differentially finitely generated algebra over the
field $K$. Consider subsets $X$ in the differential spectrum of $A$
which contain $f^*(\diffmax B)$ for any differentially finitely
generated algebra $B$ and any differential homomorphism $f\colon
A\to B$.

\begin{statement}
Among all sets $X$ specified above there is a minimal set, and this
set coincides with $\diffsmax A$.
\end{statement}
\begin{proof}
The assertion is special case of Theorem~\ref{maintheorem}.
\end{proof}

In the case of differentially finitely generated algebras,
Statements~\ref{theoremUp} and~\ref{theoremDown} can be
significantly strengthened.

\begin{statement}\label{spec_smax_state}
Let $A$ and $B$ be differential algebra over some field, and let
$f\colon A\to B$ be a differential homomorphism. Then
\begin{enumerate}
\item $f$ has the going-up property for differential ideals if and only
if $f^*_{\Delta}$ is a closed map from $\diffsmax B$ to $\diffsmax
A$;
\item $f$ has the going-down property for differential ideals if and only
if $f^*_{\Delta}A $ is an open map from $\diffsmax B$ to $\diffsmax
A$;
\item if a set $E$ is a constructible in $\diffsmax B$, then
the set $f^*_{\Delta}(E)$ is constructible in $\diffsmax A$.
\end{enumerate}
\end{statement}
\begin{proof}

First, we prove the ``only if'' part of assertion (1). If $f$ has
the going-up property for differential ideals, then
Statement~\ref{theoremUp} implies the closedness of the map
$$
f^*_\Delta\colon \diffspec B\to \diffspec A
$$
Let us show that the restriction of this map to $\diffsmax B$ is
closed. Take a differential ideal $\frak b$ in $B$ and let $\frak a$
be its contraction to $A$. If $V(\frak a)\cap\diffsmax B$ is a
closed subset of $\diffsmax B$, then its image is contained in the
set $V(\frak a)\cap\diffsmax A$. Let us show that this set coincides
with the image. Let $\frak m$ be an ideal of $V(\frak
a)\cap\diffsmax A$; then there exists a prime differential ideal
$\frak q\in V(\frak b)$. By the definition of the ideal $\frak m$,
there exists an element $s$ for which $\frak m$ is a maximal ideal
not containing $s$. Consider the set of all differential prime
ideals containing $\frak q$ and not containing $s$. By Zorn's lemma,
there exists a maximal ideal $\frak n$ with these properties. By
construction, the ideal $\frak n$ belongs to $V(\frak b)\cap
\diffsmax B$ and is contracted to $\frak m$.

We proceed to the ``if'' part. Consider the closed subset $V(\frak
b)$ in $\diffspec B$ and let $\frak a$ be a contraction of $\frak
b$. Let us show that the image of $V(\frak b)$ coincides with
$V(\frak a)$, therefore, is closed; according to
Statement~\ref{theoremUp}, this will imply the required assertion.

Let $\frak p\in V(\frak a)$. Since $\diffsmax A$ is very dense in
$\diffspec A$, it follows that
$$
\frak q=\bigcap_{\frak q\subseteq\frak m} \frak m,
$$
where the intersection is over all locally closed points of the
spectrum. For each $\frak m$, there exists an $\frak n$ in $V(\frak
b)$ contracted to $\frak m$. The intersection of all such ideals is
contracted to $\frak q$; therefore, $\frak q^{ec}=\frak q$. The
application of Statement~\ref{lemmaIdeal} completes the proof.

Now, we prove the ``only if'' part of (2). Since $f$ has the
going-down property, it follows from Statement~\ref{theoremDown}
that the map $f^*_\Delta\colon\diffspec B\to \diffspec A$ is open.
For any element $s\in A$, we have
$$
f^*_{\Delta}((\diffspec B)_s)=\bigcup_{i}(\diffspec A)_{u_i}
$$
Let us show that
$$
f^*_{\Delta}((\diffsmax B)_s)=\bigcup_{i}(\diffsmax A)_{u_i}
$$
The inclusion $\subseteq$ is obvious. Let us prove the reverse
inclusion. Suppose that $\frak m$ belongs to $(\diffsmax A)_{u_i}$.
By definition, we can find an element $t$ for which
$f^*_\Delta(\frak q)=\frak m$. Consider the set of prime
differential ideals containing $\frak q$ and not containing $st$. A
maximal element of this set has the required property. Thus, the
image of any basic open set is open; therefore, the image of any
open set is open.

Let us prove ``if'' part. Suppose that, on the contrary,
$$
f^*_{\Delta}((\diffsmax B)_s)=\bigcup_{i}(\diffsmax A)_{u_i}
$$
Let us show that
$$
f^*_{\Delta}((\diffspec B)_s)=\bigcup_{i}(\diffspec A)_{u_i}
$$
First, we prove the inclusion $\subseteq$. Take an element $\frak q$
of $(\diffspec B)_s$. It is contained in some $\frak m$ from
$(\diffsmax B)_s$; hence, the inclusion $\subseteq$ does hold.

Let us prove the inclusion $\supseteq$. According to
Statement~\ref{smaxcorrect}, any ideal
$$
\frak p\in(\diffsmax A)_{u_i}
$$
can be represented as the intersection
$$
\frak p=\bigcap_{\frak p\subseteq \frak m}\frak m
$$
over $\frak m$ belonging to $(\diffsmax A)_{u_i}$. Each $\frak m$
has a preimage $\frak n$ in $(\diffsmax B)_s$. The intersection of
these preimages is contracted to $\frak q$; hence, there exists a
prime differential ideal contracted to $\frak q$.

(3). We set $X=\diffspec A$ and $Y=\diffspec B$. Since $E$ is
constructible, it follows that it has the form
$$
E=(V(\frak b_1)\cap Y_{t_1}\cup\ldots\cup V(\frak b_n)\cap
Y_{t_n})\cap\diffsmax B.
$$
Therefore, the image of the set $V(\frak b_1)\cap
Y_{t_1}\cup\ldots\cup V(\frak b_n)\cap Y_{t_n}$ is constructible as
well (by Statement~\ref{constructtheorem}), i.~e., it has the form
$$
V(\frak a_1)\cap X_{s_1}\cup\ldots\cup V(\frak a_n)\cap X_{s_m}.
$$
Let us show that
$$
f^*_{\Delta}(E)=(V(\frak a_1)\cap X_{s_1}\cup\ldots\cup V(\frak
a_n)\cap X_{s_m})\cap \diffsmax A.
$$
Indeed, the inclusion $\subseteq$ is obvious. Let us prove the
reverse inclusion. Take a locally closed point $\frak m$ of the
differential spectrum of $A$ which belongs to the right-hand side.
There exists a prime differential ideals $\frak p$ contracted to
$\frak m$ and contained in $V(\frak b_1)\cap Y_{t_1}\cup\ldots\cup
V(\frak b_n)\cap Y_{t_n}$. Suppose that this ideal belongs to
$V(\frak b_i)\cap Y_{t_i}$ and let $s$ be an elements for which
$\frak m$ is a maximal prime differential ideal not containing $s$.
Then a maximal prime differential ideal in $B$ not containing $st_i$
is as required. The set of such ideals is nonempty, because $\frak
p$ is a prime differential ideals in $B$ not containing $s$ and
$t_i$.
\end{proof}

\section{Differential algebraic varieties}\label{sec_11}

A detailed information on differential algebraic varieties and
groups can be found in the paper~\cite{Cs}, which contains also all
facts and definitions used here without explanation. We fix a
differentially closed field $K$ of characteristic zero. It follows
from the definition of differential closedness that any simple
differentially finitely generated algebra over $K$ coincides with
$K$. As a consequence, for any differentially finitely generated
algebra $A$ over $K$, we have $\diffmax A=\diffsmax A$. On the other
hand, choosing generators in the algebra $A$, we can represent this
algebra as the quotient ring $A=K\{y_1,\ldots,y_n\}/\frak a$ of the
differential polynomial ring. The ideal $\frak a$ determines a
differential algebraic variety $V(\frak a)$ in $K^n$. The points of
$V(\frak a)$ correspond to differential homomorphisms from $A$ to
$K$. To every such homomorphism we assign its kernel, which is a
maximal differential ideal. The algebra $A$ contains the field $K$,
which is mapped isomorphically onto its quotient by the kernel;
thus, this correspondence determines a bijection between $V(\frak
a)$ and $\diffmax A$.

An example given in~\cite[Chap.~I, Sec.~5, p.~901]{Cs} shows that
regular functions on differential algebraic varieties are not
necessary polynomial. Thus, to apply the results obtained above, we
need the following final draft.

\begin{statement}\label{stat_21}
Suppose that $X$ is a differential algebraic variety, $A=K\{X\}$ is
its coordinate ring, and $f$ is a regular function on $X$. Then $X$
is isomorphic to the graph $Y$ of the function $f$.
\end{statement}
\begin{proof}

Consider the variety  $X\times K$. Its coordinate ring coincides
with the differential polynomial ring $A\{y\}$. In this set, the
graph of the function $f$ is determined by the equation $y-f=0$. The
maps $X\to Y$ and $Y\to X$ defined by $x\mapsto (x,f(x))$ and
$(x,y)\mapsto s$ are the required mutually inverse maps.
\end{proof}

\begin{corollary}
For any morphism of differential algebraic varieties  $\varphi\colon
X\to Y$, there exist coordinate rings $K\{X\}$ and $K\{Y\}$
($K\{X\}$ depends on $\varphi$) and a homomorphism $\phi\colon
K\{Y\}\to K\{X\}$ for which $\varphi=\phi^*$.
\end{corollary}
\begin{proof}

Take any coordinate rings $K\{X\}$ and $K\{Y\}$. The ring
$$
K\{Y\}=K\{y_1,\ldots,y_n\}
$$
is differentially finitely generated over $K$. Consider the
functions $x_i=\varphi^*(y_i)=y_i\circ\varphi$. These are regular
functions on $X$. According to Statement~\ref{stat_21},
$$
K\{X\}\{x_1,\ldots,x_n\}
$$
is a coordinate ring for $X$ with the required properties.
\end{proof}

It is the rime to gather in the corps. Let us translate the
statements which we proved for spectra into the language of
differential algebraic varieties and groups. The following assertion
is the translation of Theorem~\ref{maintheorem}.

\begin{statement}
If $X\to Y$ is a morphism of differential algebraic varieties and
$X$ is irreducible, then the image of $X$ contains an open subset in
its closure.
\end{statement}

Since differential algebraic variety can be identified with the set
of all locally closed points of the differential spectrum of a
coordinate ring, then, by Statement~\ref{spec_smax_state}~(3), we
have the following result.

\begin{statement}
Any morphism of differential algebraic varieties is a constructible
map.
\end{statement}

These two statements are already known (their proofs can be found
in, e.g., \cite[Chap.~1, Sec.~6, Proposition~2]{Cs}); our objective
was to show how elegantly they can be proved by using the new
technique. The following two results describe new geometric
properties of differential algebraic varieties.
Corollary~\ref{opencorollary}, together with
Statement~\ref{spec_smax_state}, implies the following assertion.

\begin{statement}\label{open_var}
If $\varphi\colon X\to Y$ is a dominant morphism of irreducible
differential algebraic variety, then the restriction of $\varphi$ to
some open set $U\subseteq X$ is open.
\end{statement}

Assertion of this kind imply interesting results in the case of
groups. Indeed, let us prove the following statement.

\begin{theorem}\label{open_group}
If $\varphi\colon G\to H$ is a dominant homomorphism of differential
algebraic groups, then $\varphi$ is an open surjective map.
\end{theorem}
\begin{proof}

the surjectivity of $\varphi$ is already known~\cite[Chap.~2,
Sec.~3, Proposition~7]{Cs}. Let us prove openness. It sufficient to
show that the restriction $\varphi\colon G_0\to H_0$ to the identity
component is open. By Statement~\ref{open_var}, there exists an open
set $U\subseteq G_0$ such that the restriction of $\varphi$ to this
set is open. The sets of the form $gU$, where $g\in G_0$, cover
$G_0$. Since multiplication by $g$ is an isomorphism, it follows
that the map $\varphi$ is open everywhere.
\end{proof}

\begin{remark}
The openness of surjective homomorphisms for differential algebraic
groups provides, in particular, a more convenient way to calculate
inverse image of sheaves on groups. Namely, let $f\colon X\to Y$ be
a continuous open map of topological spaces, and let $\mathcal F$ be
a sheaf on $Y$. Then the presheaf $f^{-1}\mathcal F$ on $X$ is
determined by $f^{-1}\mathcal F(U)=\mathcal F(f(U))$.
\end{remark}


\begin{thebibliography}{99}


\bibitem{Cs} P.J. Cassidy, ``Differential algebraic groups,'' Amer.~J.~Math. {\bf 94} (3), 891-954
(1972).


\bibitem{Kl}E.R. Kolchin, {\it Differential Algebra and Algebraic Groups}, in {\it Pure Appl. Math.}
(Academic, New York, 1973), Vol. 54.

\bibitem{Kl2}E.R. Kolchin, {\it Differential Algebraic Groups}, in {\it Pure Appl. Math.}
(Academic, Orlando, FL, 1985), Vol. 114.

\bibitem{AM}M. Atiyah and I. Macdonald, {\it Introduction to Commutative
Algebra} (Addison-Wesley, Reading, Mass., 1969; Faktorial, Moscow,
2003).

\bibitem{Kpl}I. Kaplansky, {\it An Introduction to Differential
Algebra} (Hermann, Paris, 1957; Inostrannaya Literatura, Moscow,
1959).

\bibitem{Tr}D.V. Trushin, ``The ideal of separants in the ring of differential polynomials,'' Fundam. Prikl. Mat. {\bf 13} (1), 215--227 (2007) [J. Math. Sci. (N.Y.) {\bf 152} (4), 595-603 (2008)].

\bibitem{Mu} H. Matsumura, {\it Commutative Algebra}, in {\it Math. Lecture Note Ser.} (Benjamin/Cummings, Reading, Mass., 1980), Vol.
56.

\end{thebibliography}
\end{document}